\documentclass[11pt]{amsart}
\usepackage{amssymb,amsmath,amsthm,epsfig}

\usepackage{amssymb}
\theoremstyle{plain}

\newtheorem{thm}{Theorem}[section]
\newtheorem{lem}[thm]{Lemma}
\newtheorem{cor}[thm]{Corollary}
\newtheorem{prop}[thm]{Proposition}
\theoremstyle{definition}
\newtheorem{defin}[thm]{Definition}

\def\hangbox to #1 #2{\vskip1pt\hangindent #1\noindent \hbox to #1{#2}$\!\!$}

\allowdisplaybreaks

\title{Sampling and Recovery of Multidimensional Bandlimited Functions via Frames}

\author{Benjamin Bailey}
\address{Department of Mathematics, Texas A\&M University\\  
College Station, TX 77843, USA}
\email{abailey@math.tamu.edu}

\begin{document}
\maketitle

\begin{abstract}
In this paper, we investigate frames for $L_2[-\pi,\pi]^d$ consisting of exponential functions in connection to oversampling and nonuniform sampling of bandlimited functions.  We derive a multidimensional nonuniform oversampling formula for bandlimited functions with a fairly general frequency domain.  The stability of said formula under various perturbations in the sampled data is investigated, and a computationally managable simplification of the main oversampling theorem is given.  Also, a generalization of Kadec's $1/4$ Theorem  to higher dimensions is considered. Finally, the developed techniques are used to approximate biorthogonal functions of particular exponential Riesz bases for $L_2[-\pi,\pi]$, and a well known theorem of Levinson is recovered as a corollary.
\end{abstract}

\section{Introduction}\label{S:0}

\noindent The subject of recovery of bandlimited signals from discrete data has its origins in the Whittaker-Kotel'nikov-Shannon (WKS) sampling theorem (stated below), historically the first and simplest such recovery formula. Without loss of generality, the formula recovers a function with a frequency band of $[-\pi,\pi]$ given the function's values at the integers.  The WKS theorem has drawbacks.  Foremost, the recovery formula does not converge given certain types of error in the sampled data, as Daubechies and DeVore mention in \cite{DD}.  They use oversampling to derive an alternative recovery formula which does not have this defect.  Additionally for the WKS theorem, the data nodes have to be equally spaced, and nonuniform sampling nodes are not allowed.  As discussed in \cite[pages 41-42]{Z}, nonuniform sampling of bandlimited functions has its roots in the work of Paley, Wiener, and Levinson.  Their sampling formulae recover a function from nodes $(t_n)_n$, where $(e^{i t_n x})_n$ forms a Riesz basis for $L_2[-\pi,\pi]$.  More generally, frames have been applied to nonuniform sampling, particularly in the work of Benedetto and Heller in \cite{Ben} and \cite{BenH}; see also \cite[chapter 10]{Z}.\\ 

\noindent In Section 3, we derive a multidimensional oversampling formula, (see equation (4)), for nonuniform nodes and bandlimited functions with a fairly general frequency domain;  Section 4 investigates the stability of equation (4) under perturbation of the sampled data.  Section 5 presents a computationally feasible version of equation (4) in the case where the  nodes are asymptotically uniformly distributed.  Kadec's theorem gives a criterion for the nodes $(t_n)_n$ so that $(e^{i t_n x})_n$ forms a Riesz basis for $L_2[-\pi,\pi]$.  Generalizations of Kadec's 1/4 theorem to higher dimensions are considered in Section 6, and an asymptotic equivalence of two generalizations is given.  Section 7 investigates approximation of the biorthogonal functionals of Riesz bases.  Additionally, we give a simple proof of a theorem of Levinson.\\

\noindent This paper forms a portion of the author's doctoral thesis, which is being prepared at Texas A \& M University under the direction of Thomas Schlumprecht and N. Sivakumar.

\section{Preliminaries}\label{S:1}

\noindent We use the $d$-dimensional $L_2$ Fourier transform
\begin{equation}
\mathcal{F}(f)(\cdotp) = \int_{\mathbb{R}^d} f(\xi)e^{-i \langle \cdotp , \xi \rangle} d\xi, \quad f  \in L_2(\mathbb{R})^d,\nonumber
\end{equation}
where the inverse transform is given by 
\begin{equation}
\mathcal{F}^{-1}(f)(\cdotp) = \frac{1}{(2\pi)^d}\int_{\mathbb{R}^d} f(\xi)e^{i \langle \cdotp , \xi \rangle} d\xi , \quad f  \in L_2(\mathbb{R})^d.\nonumber
\end{equation}
\noindent This is an abuse of notation.  The integral is actually a principal value where the limit is in the $L_2$ sense.  This map is an onto isomorphism from $L_2 (\mathbb{R}^d)$ to itself.\\
\begin{defin}
 Given a bounded measurable set $E$ with positive measure, we define $PW_E := \{f \in L_2(\mathbb{R}^d) \arrowvert \mathrm{supp}(\mathcal{F}^{-1}(f)) \subset E \}$.  Functions in $PW_E$ are said to be bandlimited.
\end{defin}

\begin{defin} The function $\mathrm{sinc}:\mathbb{R} \rightarrow \mathbb{R}$ is defined by $\mathrm{sinc}(x) = \frac{\sin(x)}{x}.$  We also define the multidimensional sinc function $\mathrm{SINC}:\mathbb{R}^d \rightarrow \mathbb{R}^d$ by $\mathrm{SINC}(x) = \mathrm{sinc}(x_1)\cdot\ldots\cdot\mathrm{sinc}(x_d)$, $x=(x_1,\ldots, x_d)$.
\end{defin}

\noindent We recall some basic facts about $PW_E$:\\

\noindent 1) $PW_E$ is a Hilbert space consisting of entire functions, though in this paper we only regard the functions as having real arguments.\\

\noindent 2) \noindent In $PW_E$, $L_2$ convergence implies uniform convergence.  This is an easy consequence of the Cauchy-Schwarz inequality.\\

\noindent 3) The function $\mathrm{sinc}(\pi(x-y)))$ is a reproducing kernel for $PW_{[-\pi,\pi]}$, that is, if $f \in PW_{[-\pi,\pi]}$, then we have
\begin{equation}\label{repker}
f(t) = \int_{-\infty}^{\infty} f(\tau)\mathrm{sinc}(\pi(t-\tau)) d\tau, \quad t \in \mathbb{R}.
\end{equation}

\noindent 4) The WKS sampling theorem (see for example \cite[page 91]{Y}):  If $f \in PW_{[-\pi,\pi]}$, then
\begin{equation}
f(t) = \sum_{n \in \mathbb{Z}} f(n) \mathrm{sinc}(\pi(t-n)),\quad t \in \mathbb{R},\nonumber
\end{equation}
where the sum converges in $PW_{[-\pi,\pi]}$, and hence uniformly.\\

\noindent If $(f_n)_{n \in \mathbb{N}}$ is a Schauder basis for a Hilbert space $H$, then there exists a unique set of functions $(f_n^*)_{n \in \mathbb{N}}$, (the biorthogonals of $(f_n)_{n \in \mathbb{N}}$) such that $\langle f_n , f_m^* \rangle  = \delta_{nm}.$  The biorthogonals also form a Schauder basis for $H$. Note that biorthogonality is preserved under a unitary transformation.\\

\begin{defin}
A sequence $(f_n)_n \subset H$ such that the map $ L e_n = f_n$ is an onto isomorphism is called a Riesz basis for $H$.
\end{defin}

\noindent The following definitions and facts concerning frames are found in \cite[section 4]{CA}.
\begin{defin}
A frame for a separable Hilbert space $H$ is a sequence $(f_n)_n \subset H$ such that for some $0<A<B,$
\begin{equation}\label{framedef}
A\Arrowvert f \Arrowvert^2 \leq \sum_n  | \langle f , f_n \rangle |^2 \leq B\Arrowvert f \Arrowvert^2, \quad \forall f \in H.
\end{equation}
\end{defin}

\noindent The numbers $A$ and $B$ in the equation~(\ref{framedef}) are called the lower and upper frame bounds.\\

\noindent Let $H$ be a Hilbert space with orthonormal basis $(e_n)_n$.  The following conditions are equivalent to $(f_n)_n \subset H$ being a frame for $H$.\\

\noindent 1) The map $L:H \rightarrow H$ defined by $ L e_n = f_n$ is bounded linear and onto.  This map is called the preframe operator.\\
\noindent 2) The map $L^*:H \rightarrow H$ (the adjoint of the preframe operator) given by $f \mapsto \sum_n \langle f , f_n \rangle e_n$ is an isomorphic embedding.\\

\noindent Given a frame $(f_n)_n$ with preframe operator $L$, the map $S = L L^*$  given by $S f = \sum_n \langle f , f_n \rangle f_n$ is an onto isomorphism.  $S$ is called the frame operator associated to the frame.  It follows that $S$ is positive and self-adjoint.\\

\noindent The basic connection between frames and sampling theory of bandlimited functions (more generally in a reproducing kernel Hilbert space) is straightforward.  If $(e^{i t_n (\cdot)})_n$ is a frame for $f \in PW_{[-\pi,\pi]}$ with frame operator $S$, and $f \in PW_{[-\pi,\pi]}$, then
\begin{eqnarray}
S(\mathcal{F}^{-1}(f)) = \sum_n \langle \mathcal{F}^{-1}(f) , f_n \rangle f_n = \sum_n \mathcal{F}(\mathcal{F}^{-1}(f))(t_n) f_n = \sum_n f(t_n) f_n,\nonumber
\end{eqnarray}
implying that $\mathcal{F}^{-1}(f) = \sum_n f(t_n) S^{-1} f_n$, so that $f = \sum_n f(t_n) \mathcal{F}(S^{-1}f_n).$  Note that in the case when $t_n = n$, we recover the WKS theorem.

\begin{defin}
A sequence $(f_n)_n$ satisfying the second inequality in equation~(\ref{framedef}) is called a Bessel sequence.
\end{defin}

\begin{defin}
An exact frame is a frame which ceases to be one if any of its elements is removed.
\end{defin}

\noindent It can be shown that the notions of Riesz bases, exact frames, and unconditional Schauder bases coincide.\\

\begin{defin}
A subset $S$ of $\mathbb{R}^d$ is said to be uniformly separated if
\begin{equation}
\inf_{x,y \in S, x \neq y} \Arrowvert x-y \Arrowvert_2 >0.\nonumber
\end{equation}
\end{defin}

\begin{defin}\label{seqdef}
If $S=(x_k)_k$ is a sequence of real numbers and $f$ is a function with $S$ in its domain, then $f_{S}$ denotes the sequence $(f(x_k))_k$.
\end{defin}

\section{The multidimensional oversampling theorem}\label{S:2}

\noindent In \cite{DD}, Daubechies and DeVore derive the following formula:
\begin{equation}\label{DDD}
f(t) = \frac{1}{\lambda}\sum_{n\in \mathbb{Z}}f \big( \frac{n}{\lambda} \big)g \big( t-\frac{n}{\lambda} \big),\quad t \in \mathbb{R},
\end{equation}
where $g$ is infinitely smooth and decays rapidly.  Thus oversampling allows the representation of bandlimited functions as combinations of integer translates of $g$ rather than the sinc function.  In this sense equation~(\ref{DDD}) is a generalization of the WKS theorem.  The rapid decay of $g$ yields a certain stability in the recovery formula, given bounded perturbations in the sampled data \cite{DD}.\\

\noindent In this section we derive a multidimensional version of equation~(\ref{DDD}), (Theorem 3.1) for unequally spaced sample points, and the corresponding non-oversampling version of the WKS theorem is given in Theorem 3.2.\\

\noindent Daubechies and DeVore regard $\mathcal{F}^{-1}(f)$ as an element of $L_2[-\lambda\pi,\lambda\pi]$ for some $\lambda>1$. In their proof the obvious fact that $[-\pi,\pi] \subset [-\lambda\pi,\lambda\pi]$ allows for the construction of the bump function $\mathcal{F}^{-1}(g) \in C^{\infty}(\mathbb{R})$ which is $1$ on $[-\pi,\pi]$ and $0$ off $[-\lambda\pi,\lambda\pi]$.  If their result is to be generalized to a sampling theorem for $PW_E$ in higher dimensions, a suitable condition for $E$ allowing the existence of a bump function is necessary.  If $E \subset \mathbb{R}^d$ is chosen to be compact such that for all $\lambda>1$, $E \subset \mathrm{int}(\lambda E)$, then Lemma 8.18 in \cite[page 245]{F}, a $C^\infty$-version of the Urysohn lemma, implies the existence of a smooth bump function which is $1$ on $E$ and $0$ off $\lambda E$.  It is to such regions that we generalize equation~(\ref{DDD}):\\

\begin{thm}\label{main}
Let $0 \in E \subset \mathbb{R}^d$ be compact such that for all $\lambda > 1$, $ E \subset \mathrm{int}(\lambda E).$  Choose $S=(t_n)_{n \in \mathbb{N}} \subset \mathbb{R}^d$ such that $(f_n)_{n \in \mathbb{N}}$, defined by $f_n(\cdotp) = e^{i \langle \cdotp , t_n \rangle}$, is a frame for $L_2(E)$ with frame operator $S$. Let $\lambda_0 >1$ with $\mathcal{F}^{-1}(g):\mathbb{R}^d \rightarrow \mathbb{R}$, $ \mathcal{F}^{-1}(g) \in C^\infty$ where $\mathcal{F}^{-1}(g)\arrowvert_{E} = 1$ and $\mathcal{F}^{-1}(g)\arrowvert_{(\lambda_0 E)^c}$=0.  If $\lambda \geq \lambda_0$ and $f \in PW_E$, then
\begin{equation}\label{biggie}
f(t) = \frac{1}{\lambda^d}\sum_{k \in \mathbb{N}} \Big( \sum_{n \in \mathbb{N}} B_{kn} f\big( \frac{t_n}{\lambda} \big) \Big) g\big(t-\frac{t_k}{\lambda}\big), \quad t \in \mathbb{R}^d,
\end{equation}
where $B_{kn} = \langle S^{-1}f_n , S^{-1}f_k \rangle_E$.  Convergence of the sum is in $L_2(\mathbb{R}^d)$, hence also uniform.  Further, the map $B:\ell_2(\mathbb{N}) \rightarrow \ell_2(\mathbb{N})$ defined by $(y_k)_{k \in \mathbb{N}} \mapsto \big( \sum_{n \in \mathbb{N}} B_{kn} y_n \big)_{k \in \mathbb{N}}$ is bounded linear, and is an onto isomorphism iff $(f_n)_{n \in \mathbb{N}}$ is a Riesz basis for $L_2(E)$.

\end{thm}

\begin{proof}
\noindent Define $f_{\lambda,n} (\cdotp) = f_n \big( \frac{\cdotp}{\lambda} \big)$.  Note that $(f_{\lambda,n})_n$ is a frame for $L_2(\lambda E)$ with frame operator $S_\lambda.$\\

\noindent Step 1: We show that\\
\begin{equation}\label{first}
f = \sum_{n} f \Big( \frac{t_n}{\lambda} \Big) \mathcal{F}[(S^{-1}_\lambda f_{\lambda, n})\mathcal{F}^{-1}(g)], \quad f \in PW_E.
\end{equation}
\noindent We know $\mathrm{supp}(\mathcal{F}^{-1}(f)) \subset E \subset \lambda E$, so we may work with $\mathcal{F}^{-1}(f)$ via its frame decomposition. We have
\begin{equation}
\mathcal{F}^{-1}(f) = S_\lambda^{-1} S_\lambda (\mathcal{F}^{-1} (f)) = \sum_n \langle \mathcal{F}^{-1}(f) , f_{\lambda,n} \rangle_{\lambda E} S^{-1}_{\lambda} f_{\lambda,n}, \quad \mathrm{on} \quad \lambda E\nonumber.
\end{equation}
This yields
\begin{equation}
\mathcal{F}^{-1}(f) = \sum_n \langle \mathcal{F}^{-1}(f) , f_{\lambda,n} \rangle_{\lambda E} (S^{-1}_{\lambda} f_{\lambda,n})\mathcal{F}^{-1}(g), \quad \mathrm{on} \quad \mathbb{R}^d,\nonumber
\end{equation}
\noindent since $\mathrm{supp}\mathcal{F}(g)\subset \lambda E$.  Taking Fourier transforms we obtain
\begin{equation}\label{second}
 f = \sum_n \langle \mathcal{F}^{-1}(f) , f_{\lambda,n} \rangle_{\lambda E} \mathcal{F}[(S^{-1}_{\lambda} f_{\lambda,n})\mathcal{F}^{-1}(g)], \quad \mathrm{on} \quad \mathbb{R}^d.
\end{equation}
Now 
\begin{eqnarray}
\langle \mathcal{F}^{-1}(f) , f_{\lambda,n} \rangle_{\lambda E}  =  \int_{\lambda E} \mathcal{F}^{-1}(f)(\xi) e^{-i \langle \xi , \frac{t_n}{\lambda} \rangle} d\xi = f \Big( \frac{t_n}{\lambda} \Big)\nonumber
\end{eqnarray}
which, when substituted into equation~(\ref{second}), yields (\ref{first}).\\

\noindent Step 2: We show that\\
\begin{equation}\label{third}
f(\cdotp) = \sum_n f \Big( \frac{t_n}{\lambda} \Big) \Big[ \sum_k \langle S^{-1}_\lambda f_{\lambda,n} , S^{-1}_\lambda f_{\lambda,k} \rangle_{\lambda E} g\big(\cdotp-\frac{t_k}{\lambda}\big) \Big],
\end{equation}
where convergence is in $L_2$.\\

\noindent We compute $\mathcal{F}[(S^{-1}_\lambda f_{\lambda,n})\mathcal{F}^{-1}(g)]$.  For $h \in L_2 (\lambda E)$ we have
\begin{eqnarray}
h = S_\lambda (S_\lambda^{-1} h) = \sum_k \langle S_\lambda^{-1} h , f_{\lambda,k} \rangle_{\lambda E} f_{\lambda,k} = \sum_k \langle  h , S_\lambda^{-1} f_{\lambda,k} \rangle_{\lambda E} f_{\lambda,k}.\nonumber
\end{eqnarray}
Letting $h = S^{-1}_\lambda f_{\lambda,n}$,
\begin{equation}
S^{-1}_\lambda f_{\lambda,n} = \sum_k \langle S^{-1}_\lambda f_{\lambda,n} , S^{-1}_\lambda f_{\lambda,k} \rangle_{\lambda E} f_{\lambda,k}.\nonumber
\end{equation}
This gives
\begin{eqnarray}
\mathcal{F}[(S^{-1}_\lambda f_{\lambda,n})F^{-1}(g)](\cdotp) & = & \sum_k \langle S^{-1}_\lambda f_{\lambda,n} , S^{-1}_\lambda f_{\lambda,k} \rangle_{\lambda E} \mathcal{F}[f_{\lambda,k} \mathcal{F}^{-1}(g)](\cdotp)\nonumber\\
& = & \sum_k \langle S^{-1}_\lambda f_{\lambda,n} , S^{-1}_\lambda f_{\lambda,k} \rangle_{\lambda E} \int_{\lambda E} e^{i \langle \xi , \frac{t_k}{\lambda} \rangle}\mathcal{F}^{-1}(g)(\xi) e^{-i \langle \xi , \cdotp \rangle}d \xi \nonumber \\
& = & \sum_k \langle S^{-1}_\lambda f_{\lambda,n} , S^{-1}_\lambda f_{\lambda,k} \rangle_{\lambda E} \int_{\lambda E} \mathcal{F}^{-1}(g)(\xi) e^{-i\langle \cdotp-\frac{t_k}{\lambda} , \xi \rangle} d \xi\nonumber\\ 
& = & \sum_k \langle S^{-1}_\lambda f_{\lambda,n} , S^{-1}_\lambda f_{\lambda,k} \rangle_{\lambda E} g \big( \cdotp -\frac{t_k}{\lambda} \big)\nonumber,
\end{eqnarray}
so (\ref{third}) follows from (\ref{first}).\\

\noindent Step 3: We show that\\
\begin{equation}
\langle S^{-1}_\lambda f_{\lambda,n} , S^{-1}_\lambda f_{\lambda,k} \rangle_{\lambda E} = \frac{1}{\lambda^d}\langle S^{-1} f_n , S^{-1} f_k \rangle_E, \quad \mathrm{for} \quad n,k \in \mathbb{N}.
\end{equation}

\noindent First we show $ (S^{-1}_\lambda f_{\lambda,n})(\cdotp) = \frac{1}{\lambda^d}(S^{-1} f_n)(\frac{\cdotp}{\lambda})$, or equivalently that $ f_{\lambda,n} = \frac{1}{\lambda^d}S_\lambda \big((S^{-1} f_n)(\frac{\cdotp}{\lambda})\big)$.\\

\noindent We have for any $g \in L_2 (\lambda E),$
\begin{eqnarray}
\langle g , f_{\lambda,k} \rangle_{\lambda E} & = & \int_{\lambda E} g(\xi) e ^{-i \langle \frac{\xi}{\lambda} , t_k \rangle} d \xi =  \lambda^d \int_E g(\lambda x) e ^{-i \langle x , t_k \rangle} d x =\lambda^d \langle g(\lambda(\cdot)) , f_k \rangle_E\nonumber.
\end{eqnarray}
By definition of the frame operator $S_\lambda$,
\begin{equation}
S_{\lambda} g = \sum_{k \in \mathbb{N}} \langle g , f_{\lambda,k} \rangle_{\lambda E} f_{\lambda,k}\nonumber,
\end{equation}
which then becomes
\begin{equation}
S_\lambda g = \lambda^d \sum_k \langle g(\lambda(\cdot)) , f_k \rangle_E f_{\lambda,k}.\nonumber
\end{equation}
Substituting $g = \frac{1}{\lambda^d}(S^{-1} f_n)(\frac{\cdotp}{\lambda})$ into the equation above we obtain 
\begin{eqnarray}
\frac{1}{\lambda^d} S_\lambda \big( (S^{-1} f_n) \big(\frac{\cdot}{\lambda} \big) \big) & = & \sum_k \langle S^{-1} f_n  , f_k \rangle_E f_{\lambda,k} = \big( S(S^{-1} f_n) \big) \big( \frac{\cdot}{\lambda} \big) = f_{\lambda,n}.\nonumber
\end{eqnarray}
We now compute the desired inner product:
\begin{eqnarray}
\langle S^{-1}_\lambda f_{\lambda,n} , S^{-1}_\lambda f_{\lambda,k} \rangle_{\lambda E} & = & \frac{1}{\lambda^{2d}}\int_{\lambda E} (S^{-1}f_n)\big(\frac{x}{\lambda}\big)\overline{(S^{-1}f_k)\big(\frac{x}{\lambda}\big)} dx \nonumber\\
& = &  \frac{\lambda^d}{\lambda^{2d}}\int_E (S^{-1} f_n)(x)\overline{(S^{-1} f_k)(x)} dx =  \frac{1}{\lambda^d}\langle S^{-1} f_n , S^{-1} f_k \rangle_E.\nonumber
\end{eqnarray}

\noindent Note that equation~(\ref{third}) becomes
\begin{equation}\label{fourth}
f(\cdotp) = \frac{1}{\lambda^d}\sum_n f \big( \frac{t_n}{\lambda} \big) \Big[ \sum_k \langle S^{-1} f_n , S^{-1} f_k \rangle g\big(\cdotp-\frac{t_k}{\lambda}\big) \big].
\end{equation}

\noindent Step 4:  The map $V:\ell_2 (\mathbb{N}) \rightarrow \ell_2 (\mathbb{N})$ given by $x=(x_k)_{k \in \mathbb{N}} \mapsto \big( \sum_n B_{kn} x_n \big)_{k \in \mathbb{N}} = Bx$ is bounded linear and self-adjoint.\\

\noindent  Let $(d_k)_{k \in \mathbb{N}}$ be the standard basis for $\ell_2 (\mathbb{N})$, and let $(e_k)_{k \in \mathbb{N}}$ be an orthonormal basis for $L_2(E)$.  Then
\begin{eqnarray}\label{matrixform}
V d_j & = & (B_{kj})_{k\in \mathbb{N}} = \sum_k B_{kj} d_k = \sum_k \langle S^{-1} f_j , S^{-1} f_k \rangle d_k =  \sum_k \langle L^*(S^{-1})^2 L e_j , e_k \rangle d_k,\nonumber
\end{eqnarray}
where $L$ is the preframe operator, i.e., $S =L L^*$.  Define $\phi:\ell^2(\mathbb{N}) \rightarrow L_2(E)$ by $\phi(d_k) = e_k$, $k \in \mathbb{N}$. Clearly $\phi$ is unitary.  It follows that $V = \phi^{-1} L^* (S^{-1})^2 L \phi$, which concludes Step 4.  From here on we identify $V$ with $B$. 
Clearly $B$ is an onto isomorphism iff $L$ and $L^*$ are both onto, i.e., iff the map $L e_n = f_n$ is an onto isomorphism.\\

\noindent Step 5: Verification of equation~(\ref{biggie}).  Recalling Definition \ref{seqdef}, $f_{S/\lambda} = \big( f\big(\frac{t_n}{\lambda}\big) \big)_{n \in \mathbb{N}}$; for each $t \in \mathbb{R}^d$, let $g_\lambda (t) = \big( g\big(t-\frac{t_n}{\lambda}\big) \big)_{n \in \mathbb{N}}$.  Noting that $f\big(\frac{\cdot}{\lambda} \big), g\big(t-\frac{\cdot}{\lambda}  \big) \in L_2 (\lambda E)$, and recalling that $(f_{\lambda,n})_n$ is a frame for $L_2(\lambda E)$, we have 
\begin{equation}
\sum_n \big| f\big( \frac{t_n}{\lambda} \big) \big|^2 = \sum |\langle \mathcal{F}^{-1}(f), f_{\lambda,n}\rangle_{\lambda E}|^2 \leq A_\lambda \Arrowvert \mathcal{F}^{-1}(f) \Arrowvert^2,
\end{equation}
and
\begin{equation}
\sum_n \big| g\big(t- \frac{t_n}{\lambda} \big) \big|^2 = \sum |\langle \mathcal{F}^{-1}\big( g\big(t- \frac{\cdot}{\lambda} \big)\big), f_{\lambda,n}\rangle_{\lambda E}|^2 \leq A_\lambda \Arrowvert \mathcal{F}^{-1}\big( g\big(t- \frac{\cdot}{\lambda} \big)\big) \Arrowvert^2.\nonumber
\end{equation}

\noindent Note that equation~(\ref{fourth}) becomes
\begin{eqnarray}\label{fifth}
f(t) & =  & \frac{1}{\lambda^d}\sum_n f \big( \frac{t_n}{\lambda} \big) \Big[ \sum_k B_{kn} g\big(t-\frac{t_k}{\lambda}\big) \Big]  =  \frac{1}{\lambda^d}\sum_n f \big( \frac{t_n}{\lambda} \big) \overline{\Big[ \sum_k B_{nk} \overline{g\big(t-\frac{t_k}{\lambda}\big)} \Big]}\nonumber\\
& = & \frac{1}{\lambda^d}\sum_n (f_{S/\lambda})_n (B\overline{g_\lambda (t)})_n =  \frac{1}{\lambda^d}\langle f_{S/\lambda} , B\overline{g_\lambda (t)} \rangle =  \frac{1}{\lambda^d} \langle B f_{S/\lambda} , \overline{g_\lambda (t)} \rangle\nonumber\\ & = & \frac{1}{\lambda^d}\sum_k (B f_{S/\lambda})_k g\big(t-\frac{t_k}{\lambda}\big) =  \frac{1}{\lambda^d}\sum_{k \in \mathbb{N}} \Big( \sum_{n \in \mathbb{N}} B_{kn} f\big( \frac{t_n}{\lambda} \big) \Big)\nonumber g\big(t-\frac{t_k}{\lambda}\big),
\end{eqnarray}
which proves (\ref{biggie}).\\

\noindent Step 6: We verify that convergence in equation~(\ref{biggie}) is in $L_2 (\mathbb{R})$ (hence uniform).  Define
\begin{equation}
f_n (t) = \frac{1}{\lambda^d}\sum_{1 \leq k \leq n} (B f_{S/\lambda})_k g\big(t-\frac{t_k}{\lambda}\big)\nonumber
\end{equation}
and
\begin{equation}
f_{m,n} (t) = \frac{1}{\lambda^d}\sum_{m \leq k \leq n} (B f_{S/\lambda})_k g\big(t-\frac{t_k}{\lambda}\big).\nonumber
\end{equation}
Then
\begin{eqnarray}
[\mathcal{F}^{-1}(f_{m,n})](\xi) & =  & \frac{1}{\lambda^d} \sum_{m \leq k \leq n} (B f_{S/\lambda})_k \mathcal{F}^{-1}\big[g\big(\cdotp-\frac{t_n}{\lambda} \big)  \big] \nonumber\\
& = & \frac{1}{\lambda^d} \sum_{m \leq k \leq n} (B f_{S/\lambda})_k \mathcal{F}^{-1}(g) (\xi) e^{i \langle \xi , \frac{t_k}{\lambda} \rangle},\nonumber
\end{eqnarray}
so
\begin{eqnarray}
\Arrowvert [\mathcal{F}^{-1}(f_{m,n})] \Arrowvert^2_2 & = & \frac{1}{\lambda^d} \int_{\lambda E} | \mathcal{F}^{-1}(g) (\xi)|^2 \Big| \sum_{m \leq k \leq n} (B f_{S/\lambda})_k e^{i \langle \xi , \frac{t_k}{\lambda} \rangle} \Big|^2 d \xi\nonumber\\
& \leq & \frac{1}{\lambda^d} \Big\Arrowvert \sum_{m \leq k \leq n}  (B f_{S/\lambda})_k f_{\lambda,k} \Big\Arrowvert^2_2.\nonumber
\end{eqnarray}
If $(h_n)_n$ is a orthonormal basis for $L_2(\lambda E)$, then the map $T h_k = f_{\lambda,k}$ (the preframe operator) is bounded linear, so
\begin{eqnarray}
\Arrowvert [\mathcal{F}^{-1}(f_{m,n})] \Arrowvert^2_2 & \leq & \frac{1}{\lambda^d} \Big\Arrowvert T \Big( \sum_{m \leq k \leq n} (B f_{S/\lambda})_k h_k \Big) \Big\Arrowvert^2_2 \leq  \frac{1}{\lambda^d} \Arrowvert T \Arrowvert^2 \sum_{m \leq k \leq n} | (B f_{S/\lambda})_k|^2.\nonumber
\end{eqnarray}
But $B f_{S/\lambda} \in \ell^2(\mathbb{N})$, so $\Arrowvert [\mathcal{F}^{-1}(f_{m,n})] \Arrowvert_2\rightarrow 0$ as $m,n \rightarrow \infty$.   As $\mathcal{F}^{-1}$ is an onto isomorphism, we have $\Arrowvert f_{m,n} \Arrowvert \rightarrow 0$, implying that $\Arrowvert f-f_n \Arrowvert \rightarrow 0$ as $n \rightarrow \infty$. \end{proof}

\noindent Note that equation~(\ref{main}) is conveniently written as
\begin{equation}\label{convenient}
f(t) = \frac{1}{\lambda^d}\sum_k (B f_{S/\lambda})_k g\big(t-\frac{t_k}{\lambda}\big), \quad t \in \mathbb{R}^d.
\end{equation}

\noindent Remark: There is a geometric characterization of sets $E \subset \mathbb{R}^d$ such that $E \subset \mathrm{int}(\lambda E)$ for all $\lambda >0$.  Intuitively, $E$ must be a ``continuous radial stretching of the closed unit ball''.  This is precisely formulated in the following proposition (whose proof is omitted).\\

\begin{prop}
If $0 \in E \subset \mathbb{R}^d$ is compact, then the following are equivalent:\\
1) $E \subset \mathrm{int}(\lambda E)$ for all $\lambda >1.$\\
2) There exists a continuous map $\phi:S^{d-1} \rightarrow (0, \infty)$ such that $ E = \{t y \phi(y)|y \in S^{d-1} , t \in [0,1]\}$.
\end{prop}

\noindent The following is a simplified version of Theorem~\ref{main}, which is proven in a similiar fashion:
\begin{thm}\label{hmm}
Choose $(t_n)_{n \in \mathbb{N}} \subset \mathbb{R}^d$ such that $(f_n)_{n \in \mathbb{N}}$, defined by $f_n(\cdotp) = \frac{1}{(2\pi)^{d/2}}e^{i \langle \cdotp , t_n \rangle}$, is a frame for $L_2([-\pi,\pi]^d)$.  If $f \in PW_E$, then
\begin{equation}\label{nonoversampling}
f(t) = \sum_{k \in \mathbb{N}} \Big( \sum_{n \in \mathbb{N}} B_{kn} f(t_n) \Big) \mathrm{SINC}(\pi(t-t_k)), \quad t \in \mathbb{R}^d.
\end{equation}
The matrix $B$ and the convergence of the sum are as in Theorem~\ref{main}.
\end{thm}

\noindent Equation~(\ref{biggie}) generalizes equation~(\ref{nonoversampling}) in the same way that equation~(\ref{DDD}) generalizes the WKS equation.\\

\noindent We can write equation~(\ref{nonoversampling}) as
\begin{equation}
f(t) = \sum_{k \in \mathbb{N}} (B f_S)_k \mathrm{SINC}(\pi(t-t_k)).
\end{equation}
The preceding result is similar in spirit to Theorem 1.9 in \cite[page 19]{B}.\\

\noindent Frames for $L_2 (E)$ satisfying the conditions in Theorems~\ref{main} and~\ref{hmm} occur in abundance. The following result is due to Beurling in \cite[see Theorem 1, Theorem 2, and (38)]{Be}.

\begin{thm}
Let $\Lambda \subset \mathbb{R}^d$ be countable such that
\begin{eqnarray}
r(\Lambda) & :=  & \frac{1}{2}\inf_{\lambda, \mu \in  \Lambda, \lambda \neq \mu} \Arrowvert \lambda - \mu \Arrowvert_2 >0\nonumber \\
\mathrm{and} \quad R(\Lambda) & := & \sup_{\xi \in \mathbb{R}^d} \inf_{\lambda \in \Lambda} \Arrowvert \lambda - \mu \Arrowvert_2 < \frac{\pi}{2}.\nonumber
\end{eqnarray}
If $E$ is a subset of the closed unit ball in $\mathbb{R}^d$ and $E$ has positive measure, then $\{e^{i\langle \cdot , \lambda \rangle} | \lambda \in \Lambda\}$ is a frame for $L_2(E)$.
\end{thm}

\section{Remarks regarding the stability of Theorem~\ref{main}}\label{S:3}

\noindent A desirable trait in a recovery formula is stability given error in the sampled data.  Suppose we have sample values $\tilde{f}_n = f \big( \frac{n}{\lambda} \big)+\epsilon_n$ where $\sup_{n} |\epsilon_n| =\epsilon$.  If in equation~(\ref{DDD}) we replace $f \big( \frac{n}{\lambda} \big)$ by $\tilde{f}_n$, and call the resulting expression $\tilde{f}$, then we have
\begin{equation}
|f(t)-\tilde{f}(t)| \leq \epsilon \frac{1}{\lambda}\sum_{n \in \mathbb{Z}} \Big| g\big(t -\frac{n}{\lambda} \big) \Big| \leq \epsilon (\lambda^{-1} \Arrowvert g' \Arrowvert_{L_1} + \Arrowvert g \Arrowvert_{L_1}).\nonumber
\end{equation}

\noindent It follows that equation~(\ref{DDD}) is certainly stable under $\ell_\infty$ perturbations in the data, while the WKS sampling Theorem is not.  For a more detailed discussion see \cite{DD}.\\

\noindent Such a stability result is not immediately forthcoming for equation~(\ref{biggie}), as the following example illustrates.\\

\noindent Restricting to $d=1$, let $(t_n)_{n \in \mathbb{Z}}$ satisfy $t_0 = D \notin \mathbb{Z}$, and $t_n = n$ for $n \neq 0$.  The forthcoming discussion in Section 5 shows that $(f_n)_{n \in \mathbb{Z}}$ is a Riesz basis for $L_2[-\pi,\pi]$.\\

\noindent Note that when $(f_n)_n$ is a Riesz basis, the sequence $(S^{-1} f_n)_n$ is its biorthogonal sequence.  We matrix $B$ associated to this basis is computed as follows.\\

\noindent The biorthogonal functions $(G_n)_{n \in \mathbb{Z}}$ for $(\mathrm{sinc}(\pi(\cdot-n)))_{n \in \mathbb{Z}}$ 
are
\begin{eqnarray}
G_n (t) & = &\frac{(-1)^n n (t-D) \mathrm{sinc}(\pi t)}{(n-D)(t-n)},\quad n \neq 0, \quad \mathrm{and}\nonumber\\
G_0 (t) & = &\frac{\mathrm{sinc}(\pi t)}{\mathrm{sinc}(\pi D)}.\nonumber
\end{eqnarray}
That these functions are in $PW_{[-\pi,\pi]}$ is verified by applying the Paley-Wiener Theorem \cite[page 85]{Y}, and the biorthogonality condition is verified by applying equation~(\ref{repker}). Again using equation~(\ref{repker}), we obtain
\begin{eqnarray}
& i) & \quad B_{m0} = \langle G_0, G_m \rangle = \frac{D(-1)^m}{\mathrm{sinc}(\pi D)(m-D)}, \quad m \neq 0, \nonumber\\
& ii) & \quad B_{00} = \langle G_0, G_0 \rangle = \frac{1}{\mathrm{sinc}^2 (\pi D)}, \nonumber\\
& iii) & \quad B_{mn} = \langle G_n, G_m \rangle = \delta_{nm} + \frac{D^2 (-1)^{n+m}}{(n-D)(m-D)}, \quad \mathrm{else}. \nonumber
\end{eqnarray}

\noindent Note that the rows of $B$ are not in $\ell_1$, so that as an operator acting on $\ell_\infty$, $B$ does not act boundedly. Consequently, the equation
\begin{equation}\label{perturbed}
\tilde{f}(t) = \frac{1}{\lambda}\sum_k (B \tilde{f}_{S/\lambda})_k g\big(t-\frac{t_k}{\lambda}\big)
\end{equation}
is not defined for all perturbed sequences $\tilde{f}_{S/\lambda}$ where $(\tilde{f}_{S/\lambda})_n = (f_{S/\lambda})_n + \epsilon_n$ where $\sup_{n} |\epsilon_n| =\epsilon$.\\

\noindent Despite the above failure, the following shows that there is some advantage of equation (\ref{biggie}) over equation (\ref{nonoversampling}).\\

\noindent If $\tilde{f}_{S/\lambda}$ is \textit{some} perturbation of $f_{S/\lambda}$ such that $\Arrowvert B\tilde{f}_{S/\lambda}-B f_{S/\lambda}  \Arrowvert_\infty \leq \epsilon$, then
\begin{equation}
\sup_{t \in \mathbb{R}^d}|f(t)-\tilde{f}(t)|\leq \epsilon \sum_k \Big| g\big(t-\frac{t_k}{\lambda}\big) \Big|.
\end{equation}

\section{Restriction of the sampling Theorem to the case where the exponential frame is a Riesz basis}\label{S:4}

\noindent From here on, we focus on the case where $(t_n)_{n \in \mathbb{N}}$ is an $\ell_\infty$ perturbation of the lattice $\mathbb{Z}^d$, and $(f_n)_{n \in \mathbb{N}}$ is a Riesz basis for $L_2[-\pi,\pi]^d$.  In this case, under the additional constraint that the sample nodes are asymptotically the integer lattice, the following theorem gives a computationally feasible version of equation~(\ref{biggie}) . The summands in equation~(\ref{biggie}) involves an infinite invertible matrix $B$, though under the constraints mentioned above, we show that $B$ can be replaced by a related finite-rank operator which can be computed concretely.  Precisely, one has the following.\\

\begin{thm}\label{main2}
Let $(n_k)_{k \in \mathbb{N}}$ be an enumeration of $\mathbb{Z}^d$, and $S=(t_k)_{k \in \mathbb{N}} \subset \mathbb{R}^d$ such that $$\lim_{k \rightarrow \infty} \Arrowvert n_k - t_k \Arrowvert_\infty = 0.$$  Define $e_k,f_k:\mathbb{R}^d \rightarrow \mathbb{C}$ by $e_k (x) = \frac{1}{(2\pi)^{d/2}} e^{i \langle n_k , x \rangle}$ and $\frac{1}{(2\pi)^{d/2}} e^{i \langle t_k , x \rangle}$, and let $(h_k)_k$ be the standard basis for $\ell_2(\mathbb{N})$.  Let $P_l:\ell_2(\mathbb{N}) \rightarrow \ell_2(\mathbb{N})$ be the orthogonal projection onto $\mathrm{span}\{h_1,\cdots,h_l\}$.  If $(f_k)_{k \in \mathbb{N}}$ is a Riesz basis for $L_2[-\pi,\pi]^d$, then for all $f \in PW_{[-\pi,\pi]^d}$, we have
\begin{equation}\label{truncatedequation}
f(t) = \lim_{l \rightarrow \infty} \frac{1}{\lambda^d}\sum_{k=1}^l [(P_l B^{-1} P_l)^{-1} f_{S/\lambda}]_k g\big( t-\frac{t_k}{\lambda} \big), \quad t \in \mathbb{R}^d, 
\end{equation}
where convergence is in $L_2$ and uniform.  Furthermore,
\begin{displaymath}
   (P_l B^{-1} P_l)_{nm} = \left\{
     \begin{array}{lr}
       \mathrm{sinc} \pi(t_{n,1}-t_{m,1}) \cdot \ldots \cdot \mathrm{sinc} \pi(t_{n,d}-t_{m,d}),  &  1 \leq n,m \leq l\\
       0, &  \mathrm{otherwise.}
     \end{array}
   \right.
\end{displaymath} 
Convergence of the sum is in $L_2$ and also uniform.
\end{thm}
\noindent There is a slight abuse of notation in the formula above. The matrix $P_l B^{-1} P_l$ is clearly not invertible as an operator on $\ell_2$, and it should be interpreted as the inverse of an $l \times l$ matrix acting on the first $l$ coordinates of $f_{S/\lambda}$.\\

\noindent The following version of Theorem~\ref{main2} avoids oversampling.  Its proof is similar to that of Theorem~\ref{main2}.\\

\begin{thm}\label{main3}
Under the hypotheses of Theorem \ref{main2},
\begin{equation}
f(t) = \lim_{l \rightarrow \infty}\sum_{k=1}^l [(P_l B^{-1} P_l)^{-1} f_S]_k \mathrm{SINC}( t-t_k), \quad t \in \mathbb{R}^d,
\end{equation}
where convergence of the sum is both $L_2$ and uniform.
\end{thm}

\noindent The following lemma forms the basis of the proof of the preceding theorems, as well as the other results in the paper.\\

\begin{lem}\label{finiteversion}
Let $(n_k)_{k \in \mathbb{N}}$ be an enumeration of $\mathbb{Z}^d$, and let $(t_k)_{k \in \mathbb{N}} \subset \mathbb{R}^d$.  Define $e_k,f_k:\mathbb{R}^d \rightarrow \mathbb{C}$ by $e_k (x) = \frac{1}{(2\pi)^{d/2}} e^{i \langle n_k , x \rangle}$ and $f_k(x) = \frac{1}{(2\pi)^{d/2}} e^{i \langle t_k , x \rangle}$. Then for any $r,s\geq 1$, and any finite sequence $(a_k)_{k=r}^s$, we have
\begin{equation}\label{norm}
\Bigg\Arrowvert \sum_{k=r}^s\Big( \frac{a_k}{(2\pi)^{d/2}}e^{i\langle (\cdot) , n_k \rangle}- \frac{a_k}{(2\pi)^{d/2}}e^{i\langle (\cdot) , t_k \rangle} \Big) \Bigg\Arrowvert_2 \leq \Big( e^{\pi d \big({\sup \atop {r \leq k \leq s}} \Arrowvert n_k-t_k \Arrowvert_\infty\big)}-1\Big) \Big(\sum_{k=r}^s |a_k|^2\Big)^{1/2}.
\end{equation}
\end{lem}

\begin{proof}
Let $\delta_k = t_k-n_k$ where $\delta_k = (\delta_{k1}, \cdots , \delta_{kd})$.  Then
\begin{eqnarray}\label{blork}
\phi_{r,s}(x):=\sum_{k=r}^s \frac{a_k}{(2\pi)^{d/2}} \big[ e^{i\langle n_k , x \rangle} - e^{i\langle t_k , x \rangle}\big]
 =  \sum_{k=r}^s \frac{a_k}{(2\pi)^{d/2}}e^{i\langle n_k , x \rangle} \big[ 1 - e^{i\langle \delta_k , x \rangle}\big],
\end{eqnarray}
Now for any $\delta_k$,
\begin{eqnarray}
1 - e^{i \langle \delta_k , x \rangle} & = & 1-e^{i\delta_{k1} x_1} \cdot \ldots \cdot e^{i\delta_{kd} x_d} =  1-\Big(\sum_{j_1 = 0}^{\infty} \frac{(i \delta_{k1} x_1)^{j_1}}{j_1 !}  \Big) \cdot \ldots \cdot \Big( \sum_{j_d = 0}^{\infty} \frac{(i \delta_{kd} x_d)^{j_d}}{j_d !} \Big)\nonumber\\
& = & 1-\sum_{(j_1,\cdots, j_d) \atop j_i \geq 0} \frac{(i \delta_{k1} x_1)^{j_1} \cdot \ldots \cdot (i \delta_{kd} x_d)^{j_d}}{j_1 ! \cdot \ldots \cdot j_d !}\nonumber\\
& = & -\sum_{(j_1,\cdots, j_d) \in J} i^{j_1 +\ldots + j_d}\frac{(\delta_{k1} x_1)^{j_1} \cdot \ldots \cdot (\delta_{kd} x_d)^{j_d}}{j_1 ! \cdot \ldots \cdot j_d !},\nonumber
\end{eqnarray}
where $J = \{ (j_1,\cdots, j_d) \in \mathbb{Z}^d |j_i \geq 0, (j_1,\cdots, j_d) \neq 0\}$.  Then equation~(\ref{blork}) becomes
\begin{eqnarray}
\phi_{r,s}(x) & = & -\sum_{k=r}^s \frac{a_k}{(2\pi)^{d/2}}e^{i\langle n_k , x \rangle}\Big[ \sum_{(j_1,\cdots, j_d) \in J} i^{j_1 +\ldots + j_d}\frac{(\delta_{k1} x_1)^{j_1} \cdot \ldots \cdot (\delta_{kd} x_d)^{j_d}}{j_1 ! \cdot \ldots \cdot j_d !} \Big]\nonumber\\
& = & -\sum_{(j_1,\cdots, j_d) \in J} \frac{x_1^{j_1} \cdot \ldots \cdot x_d^{j_d}}{j_1 ! \cdot \ldots \cdot j_d !} i^{j_1 +\ldots + j_d} \sum_{k=r}^s\frac{a_k}{(2\pi)^{d/2}}\delta_{k1}^{j_1}  \cdot \ldots \cdot \delta_{kd}^{j_d} e^{i \langle n_k , x \rangle},\nonumber
\end{eqnarray}
so
\begin{eqnarray}
|\phi_{r,s}(x)| \leq \sum_{(j_1,\cdots, j_d) \in J} \frac{\pi^{j_1 + \ldots +j_d}}{j_1 ! \cdot \ldots \cdot j_d !} \Big| \sum_{k=r}^s a_k \delta_{k1}^{j_1}  \cdot \ldots \cdot \delta_{kd}^{j_d} \frac{e^{i \langle n_k , x \rangle}}{(2\pi)^{d/2}} \Big|.\nonumber
\end{eqnarray}
For brevity denote the outer summand above by $h_{j_1 , \ldots ,j_d}(t)$.  Then
\begin{eqnarray}
\bigg(\int_{[-\pi,\pi]^d} |\phi_{r,s}(x)|^2 dt \bigg)^\frac{1}{2} & \leq & \bigg( \int_{[-\pi,\pi]^d} \Big| \sum_{(j_1,\cdots, j_d) \in J } h_{j_1 , \ldots ,j_d}(x) \Big|^2 dx \bigg)^\frac{1}{2} \nonumber\\
& \leq & \sum_{(j_1,\cdots, j_d) \in J } \bigg( \int_{[-\pi,\pi]^d} \Big| h_{j_1 , \ldots ,j_d}(x) \Big|^2 dx \bigg)^\frac{1}{2},\nonumber
\end{eqnarray}
so that
\begin{eqnarray}
\Arrowvert \phi_{r,s} \Arrowvert_2 & \leq & \sum_{(j_1,\cdots, j_d) \in J } \frac{\pi^{j_1 + \cdot \ldots \cdot +j_d}}{j_1 ! \cdot \ldots \cdot j_d !} \Big( \int_{[-\pi,\pi]^d} \bigg| \sum_{k=r}^s a_k \delta_{k1}^{j_1}  \cdot \ldots \cdot \delta_{kd}^{j_d} \frac{e^{i \langle n_k , x \rangle}}{(2\pi)^{d/2}} \bigg|^2 dx \Big)^\frac{1}{2}\nonumber\\
& = & \sum_{(j_1,\cdots, j_d) \in J }  \frac{\pi^{j_1 + \cdot \ldots \cdot +j_d}}{j_1 ! \cdot \ldots \cdot j_d !} \Big( \sum_{k=r}^s |a_k|^2 |\delta_{k1}^{j_1}|^2 \cdot \ldots \cdot |\delta_{kd}^{j_d}|^2 \Big)^\frac{1}{2}\nonumber\\
& \leq & \sum_{(j_1,\cdots, j_d) \in J }  \frac{\pi^{j_1 + \cdot \ldots \cdot +j_d}}{j_1 ! \cdot \ldots \cdot j_d !} \Bigg( \sum_{k=r}^s |a_k|^2 \Big({\sup \atop {r \leq k \leq s}} \Arrowvert n_k-t_k \Arrowvert_\infty\Big)^{2(j_1 + \ldots + j_d)} \Bigg)^\frac{1}{2}\nonumber\\
& = & \sum_{(j_1,\cdots, j_d) \in J }  \frac{\Big(\pi{\sup \atop {r \leq k \leq s}} \Arrowvert n_k-t_k \Arrowvert_\infty\Big)^{j_1 + \cdot \ldots \cdot + j_d}}{j_1 ! \cdot \ldots \cdot j_d !} \Big( \sum_{k=r}^s |a_k|^2 \Big)^\frac{1}{2}\nonumber\\
& = & \bigg[ \prod_{l=1}^d\bigg( \sum_{j_\ell = 0}^{\infty} \frac{\big(\pi{\sup \atop {r \leq k \leq s}} \Arrowvert n_k-t_k \Arrowvert_\infty\big)^{j_\ell}}{j_\ell !} \bigg)-1 \bigg]\Big(\sum_{k=r}^s |a_k|^2\Big)^\frac{1}{2}\nonumber\\
& = & \Big( e^{\pi d \big({\sup \atop {r \leq k \leq s}} \Arrowvert n_k-t_k \Arrowvert_\infty\big)}-1\Big) \Big(\sum_{k=r}^s |a_k|^2\Big)^{\frac{1}{2}}.\nonumber
\end{eqnarray}
\end{proof}

\begin{cor}\label{estimate}
Let $(n_k)_{k \in \mathbb{N}}$ be an enumeration of $\mathbb{Z}^d$, and let $(t_k)_{k \in \mathbb{N}} \subset \mathbb{R}^d$ such that $$\sup_{k \in \mathbb{N}} \Arrowvert n_k - t_k \Arrowvert_\infty = L < \infty.$$  Define $e_k,f_k:\mathbb{R}^d \rightarrow \mathbb{C}$ by $e_k (x) = \frac{1}{(2\pi)^{d/2}} e^{i \langle n_k , x \rangle}$ and $\frac{1}{(2\pi)^{d/2}} e^{i \langle t_k , x \rangle}$.  Then the map $T:L_2[-\pi,\pi]^d \rightarrow L_2[-\pi,\pi]^d$, defined by $T e_n = e_n - f_n$, satisfies the following estimate:
\begin{equation}\label{norm2}
\Arrowvert T \Arrowvert \leq e^{\pi L d} -1.
\end{equation}
\end{cor}
\begin{proof}
Lemma (\ref{finiteversion}) shows that $T$ is uniformly continuous on a dense subset of the ball in $L_2(E)$, so $T$ is bounded on $L_2[-\pi,\pi]^d$. The inequality~(\ref{norm2}) follows immediately. 
\end{proof}

\begin{cor}\label{truncate}
 Let $(n_k)_{k \in \mathbb{N}}$, $(t_k)_{k \in \mathbb{N}} \subset \mathbb{R}^d$, and let $e_k$, $f_k$ and $T$ be defined as in Corollary \ref{estimate}.  For each $l \in \mathbb{N}$, define $T_l$ by $T_l e_k = e_k - f_k$ for $1 \leq k \leq l$, and $T_l e_k = 0$ for $ l < k$.  If $\lim_{k \rightarrow\infty} \Arrowvert n_k - t_k \Arrowvert_\infty = 0$, then $ \lim_{l \rightarrow \infty} T_l = T\nonumber$ in the operator norm.  In particular, $T$ is a compact operator.
\end{cor}
\begin{proof}
As
\begin{eqnarray}
(T-T_l) \big( \sum_{k=1}^\infty a_k e_k \big) & = & \sum_{k=1}^\infty a_k (e_k -f_k) - \sum_{k=1}^l a_k (e_k -f_k) \nonumber\\ 
& = & \sum_{k=l+1}^\infty a_k (e_k -f_k) = T \big( \sum_{k=l+1}^\infty a_k e_k \big),\nonumber
\end{eqnarray}
the estimate derived in lemma (\ref{finiteversion}) yields
\begin{eqnarray}
\big\Arrowvert (T-T_l)\big( \sum_{k=1}^\infty a_k e_k \big) \big\Arrowvert_2 & = & \big\Arrowvert T \big( \sum_{k=l+1}^\infty a_k e_k \big) \big\Arrowvert_2 \leq  \big( e^{\pi d {\sup \atop k \geq l+1} \Arrowvert \delta_k \Arrowvert_\infty}-1 \big) \big\Arrowvert \sum_{k=1}^\infty a_k e_k \big\Arrowvert_2,\nonumber
\end{eqnarray}
so $\big\Arrowvert (T-T_l)\big\Arrowvert_2 \rightarrow 0$ as $l \rightarrow \infty$. As $T_l$ has finite rank, we deduce that $T$ is compact. \end{proof}

\noindent We are ready for the proof of Theorem~\ref{main2}.\\

\begin{proof}


\noindent Step 1:  $B$ is a compact perturbation of the identity map, namely
\begin{equation}\label{comppert}
B =  I+\lim_{l \rightarrow \infty} (-P_l +(P_l B^{-1} P_l)^{-1}).
\end{equation}

\noindent Since $(f_k)_{k \in \mathbb{N}}$ is a Riesz basis for $L_2[-\pi,\pi]^d$, $L^* = (I-T)$ is an onto isomorphism where $T e_k = e_k -f_k$; so $B$ simplifies to $(I-T)^{-1} (I-T^*)^{-1}$.  We examine 
\begin{eqnarray}
B^{-1} = (I-T^*)(I-T) = I+(T^* T-T-T^*) := I+\Delta,\nonumber
\end{eqnarray}
where $\Delta$ is a compact operator.  If an operator $\Delta:H \rightarrow H$ is compact then so is $\Delta^*$, hence $P_l \Delta P_l \rightarrow \Delta$ in the operator norm because
\begin{eqnarray}
\Arrowvert P_l \Delta P_l -\Delta  \Arrowvert & \leq & \Arrowvert P_l \Delta P_l -P_l \Delta \Arrowvert + \Arrowvert P_l \Delta -\Delta \Arrowvert \leq  \Arrowvert \Delta P_l -\Delta \Arrowvert + \Arrowvert P_l \Delta -\Delta \Arrowvert\nonumber\\ & = & \Arrowvert  P_l \Delta^* -\Delta^* \Arrowvert + \Arrowvert P_l \Delta -\Delta \Arrowvert \rightarrow 0.\nonumber
\end{eqnarray}
We have
\begin{eqnarray}
B^{-1} & = & \lim_{l \rightarrow \infty} (I+P_l \Delta P_l) =  \lim_{l \rightarrow \infty} (I+P_l (B^{-1}-I) P_l) =  \lim_{l \rightarrow \infty} (I-P_l +P_l B^{-1} P_l).\nonumber
\end{eqnarray}
Now $(P_l B^{-1} P_l)$ restricted to the first $l$ rows and columns is the Grammian matrix for the set $(f_1,\cdots,f_l)$ which can be shown (in a straightforward manner) to be linearly independent. We conclude that $P_l B^{-1} P_l$ is invertible as an $l \times l$ matrix.  By $(P_l B^{-1} P_l)^{-1}$ we mean the inverse as an $l \times l$ matrix and zeroes elsewhere.  Observing that the ranges of $P_l B^{-1} P_l$ and $(P_l B^{-1} P_l)^{-1}$ are in the kernel of $I-P_l$, and that the range of $I-P_l$ is in the kernels of $P_l B^{-1} P_l$ and $(P_l B^{-1} P_l)^{-1}$, we easily compute
\begin{equation}
(I-P_l +(P_l B^{-1} P_l)^{-1})^{-1} = I-P_l +P_l B^{-1} P_l,\nonumber
\end{equation}
so that 
\begin{equation}
B^{-1} = \lim_{l \rightarrow \infty} (I-P_l +(P_l B^{-1} P_l)^{-1})^{-1},\nonumber
\end{equation}
implying
\begin{eqnarray}
B & = & \lim_{l \rightarrow \infty} (I-P_l +(P_l B^{-1} P_l)^{-1}) := \lim_{l \rightarrow \infty} B_l =  I+\lim_{l \rightarrow \infty} (-P_l +(P_l B^{-1} P_l)^{-1}).\nonumber
\end{eqnarray}

\noindent Step 2: We verifiy equation~(\ref{truncatedequation}) and its convergence properties. Recalling equation~(\ref{convenient}), we have
\begin{eqnarray}
f(t) & - & \frac{1}{\lambda^d} \sum_{k=1}^{\infty}[(I-P_l +(P_l B^{-1} P_l)^{-1}) f_{S/\lambda}]_k g\big( t- \frac{t_k}{\lambda} \big) = \frac{1}{\lambda^d} \sum_{k=1}^{\infty}[(B-B_l) f_{S/\lambda}]_k g\big( t- \frac{t_k}{\lambda} \big)\nonumber
\end{eqnarray}
implying
\begin{eqnarray}
f(t) & - & \frac{1}{\lambda^d} \sum_{k=1}^{l}[(P_l B^{-1} P_l)^{-1} f_{S/\lambda}]_k g\big( t- \frac{t_k}{\lambda} \big) \nonumber\\ & = & \frac{1}{\lambda^d} \sum_{k=1}^{\infty}[(B-B_l) f_{S/\lambda}]_k g\big( t- \frac{t_k}{\lambda} \big) + \frac{1}{\lambda^d}\sum_{k=l+1}^{\infty} f\big( \frac{t_k}{\lambda} \big) g\big( t- \frac{t_k}{\lambda} \big).\nonumber
\end{eqnarray}
Therefore,
\begin{eqnarray}
& \Big\Arrowvert & f(\cdot) - \frac{1}{\lambda^d} \sum_{k=1}^{l}[(P_l B^{-1} P_l)^{-1} f_{S/\lambda}]_k g\big( \cdot- \frac{t_k}{\lambda} \big) \Big\Arrowvert_2 = \nonumber\\
& = & \Big\Arrowvert \frac{1}{\lambda^d} \sum_{k=1}^{\infty}[(B-B_l) f_{S/\lambda}]_k g\big( \cdot- \frac{t_k}{\lambda} \big) + \frac{1}{\lambda^d} \sum_{k=l+1}^{\infty} f\big( \frac{t_k}{\lambda} \big) g\big( \cdot- \frac{t_k}{\lambda} \big) \Big\Arrowvert_{[-\lambda\pi,\lambda\pi]^d} \nonumber\\
& = & \frac{1}{\lambda^d} \Big\Arrowvert \mathcal{F}^{-1}(g)(\cdot) \Big( \sum_{k=1}^{\infty}[(B-B_l) f_{S/\lambda}]_k e^{i\langle \cdot , \frac{t_k}{\lambda} \rangle} + \sum_{k=l+1}^{\infty} f\big( \frac{t_k}{\lambda} \big) e^{i\langle \cdot , \frac{t_k}{\lambda} \rangle} \Big) \Big\Arrowvert_{[-\lambda\pi,\lambda\pi]^d}\nonumber
\end{eqnarray}
after taking the inverse Fourier transform.  Now
\begin{eqnarray}
& \Big\Arrowvert & f(\cdot) - \frac{1}{\lambda^d} \sum_{k=1}^{l}[(P_l B^{-1} P_l)^{-1} f_{S/\lambda}]_k g\big( \cdot- \frac{t_k}{\lambda} \big) \Big\Arrowvert_2 \nonumber\\
& \leq & \frac{1}{\lambda^d} \Big\Arrowvert \sum_{k=1}^{\infty}[(B-B_l) f_{S/\lambda}]_k e^{i\langle \cdot , \frac{t_k}{\lambda} \rangle} \Big\Arrowvert_{[-\lambda\pi,\lambda\pi]^d} + \frac{1}{\lambda^d} \Big\Arrowvert \sum_{k=l+1}^{\infty} f\big( \frac{t_k}{\lambda} \big) e^{i\langle \cdot , \frac{t_k}{\lambda} \rangle} \Big\Arrowvert_{[-\lambda\pi,\lambda\pi]^d}\nonumber\\
& \leq & \frac{M}{\lambda^d} \Big\Arrowvert (B-B_l)f_{S/\lambda} \Big\Arrowvert_{\ell_2(\mathbb{N})} + \frac{M}{\lambda^d} \Big( \sum_{k=l+1}^{\infty} \big| f\big( \frac{t_k}{\lambda} \big)  \big|^2 \Big)^\frac{1}{2},\nonumber
\end{eqnarray}
since  $\big(f_k\big(\frac{\cdot}{\lambda}\big)\big)_k$ is a Riesz basis for $L_2[-\lambda\pi, \lambda\pi]^d$.  Since $B_l \rightarrow B$ as $l \rightarrow \infty$ and $\big( f\big(\frac{t_k}{\lambda}\big) \big)_k \in \ell_2(\mathbb{N})$, the last two terms in the inequality above tend to zero, which proves the required result.\\

\noindent Finally, to compute $(P_l B^{-1} P_l)_{nm}$,  recall that $B^{-1} = (I-T^*)(I-T)$.  Proceeding in a manner similar to the proof of equation~(\ref{matrixform}), we obtain
\begin{eqnarray}
B^{-1}_{mn} & = & \langle L L^* e_n , e_m \rangle = \langle L^* e_n ,L^* e_m \rangle = \langle f_n , f_m \rangle\nonumber\\
& = & \mathrm{sinc} \pi(t_{n,1}-t_{m,1}) \cdot \ldots \cdot \mathrm{sinc} \pi(t_{n,d}-t_{m,d}).\nonumber
\end{eqnarray}
The entries of $P_l B^{-1} P_l$ agree with those of $B^{-1}$ when $1 \leq n,m \leq l$.
\end{proof}

\noindent One generalization of Kadec's $1/4$ theorem given by Pak and Shin in \cite{PS} (which is actually a special case of Avdonin's theorem) is:
\begin{thm}\label{pakshin}
Let $(t_n)_{n \in \mathbb{Z}} \subset \mathbb{R}$ be a sequence of distinct points such that $$\limsup_{|n| \rightarrow \infty} |n-t_n| = L < \frac{1}{4}.$$  Then the sequence of functions $(f_k)_{k \in \mathbb{Z}}$, defined by $f_k (x) = \frac{1}{\sqrt{2 \pi}}e^{i t_k x}$, is a Riesz basis for $L_2[-\pi,\pi]$.
\end{thm}

\noindent Theorem~\ref{pakshin} shows that in the univariate case of Theorem \ref{main2}, the restriction that $(f_k)_{k \in \mathbb{N}}$ is a Riesz basis for $L_2[-\pi,\pi]$ can be dropped.  The following example shows that the multivariate case is very different\\

\noindent Let $(e_n)_n$ be an orthonormal basis for a Hilbert space $H$.  Let $f_1 \in H$ with $\Arrowvert f_1 \Arrowvert = 1$, then $(f_1,e_2, e_3, \cdots)$ is a Riesz basis for $H$ iff $\langle f_1 , e_1 \rangle \neq 0$.  Verifying that the map $T$, given by $e_k \mapsto e_k$ for $k>1$ and $e_1 \mapsto f_1$, is a continuous bijection is routine, so $T$ is an isomorphism via the Open Mapping Theorem.  In the language of Theorem~\ref{main2}, $(f_1,e_2, e_3, \cdots)$ is a Riesz basis for $L_2 [-\pi, \pi]$ iff $$ 0 \neq \mathrm{sinc}(\pi t_{1,1})\cdot \ldots \cdot \mathrm{sinc}(\pi t_{1,d}),$$ that is, iff $t_1 \in (\mathbb{R} \setminus \{ \pm 1, \pm 2, \cdots \})^d$.\\



\section{Generalizations of Kadec's 1/4 Theorem}\label{S:5}

\noindent Corollary \ref{estimate} yields the following generalization of Kadec's Theorem in $d$ dimensions.

\begin{cor}\label{kadeclike}
Let $(n_k)_{k \in \mathbb{N}}$ be an enumeration of $\mathbb{Z}^d$, and let $(t_k)_{k \in \mathbb{N}} \subset \mathbb{R}^d$ such that
\begin{equation}\label{ineq}
\sup_{k \in \mathbb{N}} \Arrowvert n_k - t_k \Arrowvert_\infty = L < \frac{\ln(2)}{\pi d}.
\end{equation}
Then the sequence $(f_k)_{k \in \mathbb{N}}$ defined by $f_k (x) = \frac{1}{(2\pi)^{d/2}}e^{i \langle x , t_k \rangle}$ is a Riesz basis for $L_2[-\pi,\pi]^d$.
\end{cor}
\noindent The proof is immediate.  Note that equation~(\ref{norm2}) implies that the map $T$ given in Corollary \ref{estimate} has norm less than $1$.  We conclude that the map $(I-T)e_k = f_k$ is invertible by considering its Neumann series.\\

\noindent The proof of Corollary~(\ref{estimate}) and Corollary~(\ref{kadeclike}) are straightforward generalizations of the univariate result proved by Duffin and Eachus \cite{DE}. Kadec improved the value of the constant in the inequality~(\ref{ineq}) (for $d=1$) from $\frac{\ln(2)}{\pi}$ to the optimal value of 1/4; this is his celebrated ``1/4 theorem'' \cite{Kad}.\\

\noindent Kadec's method of proof is to expand $e^{i \delta x}$ with respect to the orthogonal basis $$\{1,\cos(nx),\sin \big( n-\frac{1}{2} \big)x  \}_{n \in \mathbb{N}}$$ for $L_2[-\pi,\pi]$, and use this expansion to estimate the norm of $T$.  In the proof of Corollary~(\ref{estimate}) and Corollary~(\ref{kadeclike}) we simply used a Taylor series.  Unlike the estimates in Kadec's Theorem, the estimate in equation~(\ref{norm2}) can be used for any sequence $(t_k)_{k \in \mathbb{N}} \subset \mathbb{R}^d$ such that $\sup_{k \in \mathbb{N}} \Arrowvert n_k - t_k \Arrowvert_\infty = L < \infty$, not only those for which the exponentials $(e^{i t_n x})_n$ form a Riesz basis. An impressive generalization of Kadec's 1/4 theorem when $d=1$ is Avdonin's ``1/4 in the mean'' theorem, \cite{AV}.\\

\noindent Sun and Zhou (see \cite{SZ} second half of Theorem 1.3) refined Kadec's argument to obtain a partial generalization of his result in higher dimensions:\\

\begin{thm}\label{sunzhou}
Let $(a_n)_{n \in \mathbb{Z}^d} \subset \mathbb{R}^d$ such that
$$0  <  L < \frac{1}{4},$$
$$D_d(L) := \Big(1-\cos \pi L+\sin \pi L  +\frac{\sin \pi L}{\pi L} \Big)^d -\Big( \frac{\sin \pi L}{\pi L} \Big)^d,\quad \mathrm{and}$$
$$\Arrowvert a_n - n \Arrowvert_\infty \leq L, \quad n \in \mathbb{Z}^d.$$
If $D_d(L)<1$, then $\big(\frac{1}{(2 \pi)^d} e^{i \langle a_n , (\cdot) \rangle}\big)$ is a Riesz basis for $L_2[-\pi,\pi]^d$ with frame bounds $(1-D_d(L))^2$ and $(1+D_d(L))^2$. 
\end{thm}

\noindent In the one-dimensional case, Kadec's theorem is recovered exactly from Theorem \ref{sunzhou}, When $d>1$, the value $x_d$ satisfying $0 < x_d <1/4$ and $D_d(x_d)=1$ is an upper bound for any value of $L$ satisfying  $0 <L < 1/4$ and $D_d(L) < 1$.  The value of $x_d$ is not readily apparent, whereas the constant in Corollary \ref{kadeclike} is $\frac{\ln 2}{\pi d}$.  A relationship between this number and $x_d$ is given in the following theorem (whose proof is omitted).
\begin{thm}
Let $x_d$ be the unique number satisfying $0 < x_d <1/4$ and $D_d(x_d)=1$.  Then
$$\lim_{d \rightarrow \infty} \frac{x_d -\frac{\ln 2}{\pi d}}{\frac{(\ln 2)^2}{12 \pi d^2}} =1. $$
\end{thm}

\noindent Thus, for sufficiently large $d$, Theorem \ref{sunzhou} and Corollary \ref{kadeclike} are essentially the same.\\

\section{A method of approximation of biorthogonal functions and a recovery of a theorem of Levinson}\label{S:6}

\noindent In this section we apply the techniques developed in the previously to approximate the biorthogonal functions to Riesz bases $\big(\frac{1}{\sqrt{2\pi}}e^{i t_n (\cdot)}\big)$ for which the preframe operator is small perturbation of the identity.  This is the content of Theorem 7.1.  A well known theorem of Levinson (see \cite[pages 47-67]{L}), follows as a corollary to Theorem 7.1.\\

\begin{defin}
A Kadec sequence is a sequence $(t_n)_{n \in \mathbb{Z}}$ of real numbers satisfying $$\sup_{n \in \mathbb{Z}} |t_n -n| = D < 1/4.$$
\end{defin}

\begin{thm}\label{levinsonlike}
Let $(t_n)_{n \in \mathbb{Z}}\subset \mathbb{R}$ be a sequence (with $t_n \neq 0$ for $n \neq 0$) such that $ (f_n)_n = \big(\frac{1}{\sqrt{2\pi}} e^{i t_n (\cdot)} \big)_n$ is a Riesz basis for $L_2[-\pi,\pi]$, and let $(e_n)_n$ be the standard exponential orthonormal basis for $L_2[-\pi,\pi]$.  If the map $L$ given by $L e_n = f_n$ satisfies the estimate $\Arrowvert I-L \Arrowvert<1,$ then the biorthogonals $G_n$ of $\frac{1}{\sqrt{2\pi}}\mathcal{F}(f_n)(\cdot) = \mathrm{sinc}(\pi(\cdot-t_n))$ in $PW_{[-\pi,\pi]}$ are
\begin{equation}
G_n (t) = \frac{H(t)}{(t-t_n)H^{'}(t_n)}, \quad n \in \mathbb{Z},
\end{equation}
where
\begin{equation}\label{holomorphic}
H(t) = (t-t_0) \prod_{n=1}^{\infty} \Big(1-\frac{t}{t_n} \Big) \Big(1-\frac{t}{t_{-n}} \Big).
\end{equation}
\end{thm}

\begin{defin}
Let $(t_n)_{n \in \mathbb{Z}}\subset \mathbb{R}$ be a sequence such that $(f_n)_n = \big(\frac{1}{\sqrt{2\pi}} e^{i t_n (\cdot)} \big)_n$ is a Riesz basis for $L_2[-\pi,\pi]$.    If $l \geq 0$, the $l$-truncated sequence $(t_{l,n})_{n \in \mathbb{Z}}$ is defined by $t_{l,n} = t_n$ if $|n| \leq l$ and $t_{l,n} = n$ otherwise.  Define $f_{l,n} = \frac{1}{\sqrt{2\pi}}e^{i t_{l,n}(\cdot)}$ for $n \in \mathbb{Z},$ $l \geq 0$.
\end{defin}

\noindent Let $P_l: L_2[-\pi,\pi] \rightarrow L_2[-\pi,\pi]$ be the orthogonal projection onto $\mathrm{span}\{e_{-l},\ldots, e_l \}$.\\

\begin{prop}\label{asdf}
Let $(t_n)_{n \in \mathbb{Z}}\subset \mathbb{R}$ be a sequence such that $(f_n)_n$ (defined above) is a Riesz basis for $L_2[-\pi,\pi]$. If $(e_n)_n$ is the standard exponential orthonormal basis for $L_2[-\pi,\pi]$ and the map $L$ (defined above) satisfies the estimate $\Arrowvert I-L \Arrowvert = \delta <1$, then the following are true:\\

\noindent 1) For $l \geq 0$ , the sequence $(f_{l,n})_n$ is a Riesz basis for $L_2[-\pi,\pi]$.\\
\noindent 2) For $l \geq 0$, the map $L_l$ defined by $L_l e_n = f_{l,n}$ satisfies $ \Arrowvert L_l^{-1} \Arrowvert \leq \frac{1}{1-\delta}.$

\end{prop}

\begin{proof}
If $(c_n)_n \in \ell_2 (\mathbb{Z})$, then
\begin{eqnarray}
(I-L_l)\big(\sum_n c_n e_n \big) & = & \sum_n c_n (e_n-L_l e_n)  = \sum_{|n|\leq l} (e_n -f_n)
= (I-L) P_l \big(\sum_n c_n e_n \big)\nonumber,
\end{eqnarray}
so that
\begin{equation}\label{asdfg}
(I-L_l) = (I-L)P_l.
\end{equation}
From this, $\Arrowvert I-L_l \Arrowvert \leq \delta$, which implies 1) and 2).
\end{proof}

\noindent Define the biorthogonal functions of $(f_{l,n})_n$ to be $(f_{l,n}^*)_n$.  Passing to the Fourier transform, we have $\frac{1}{\sqrt{2\pi}}\mathcal{F}(f_{l,n})(t) = \mathrm{sinc}(\pi(t-t_{l,n}))$ and $G_{l,n} (t) := \frac{1}{\sqrt{2\pi}}\mathcal{F}(f_{l,n}^*)(t)$.  Define the biorthogonal functions of $(f_n)_n$ similarly.\\

\begin{lem}\label{biorthapprox}
If $(t_n)_n \subset \mathbb{R}$ satisfies the hypotheses of proposition \ref{asdf}, then $$\lim_{l \rightarrow \infty} G_{l,n} = G_n$$ in $PW_{[-\pi,\pi]}$.
\end{lem}

\begin{proof}
Note that
\begin{eqnarray}
\delta_{nm} = \langle f_{l,n} , f_{l,m}^* \rangle = \langle L_l e_n , f_{l,m}^* \rangle = \langle e_n , L_l^* f_{l,m}^*  \rangle\nonumber
\end{eqnarray}
so that for all $m$, $f_{l,m}^* = (L_l^*)^{-1}e_m$.  Similarly, $f_m^* = (L^*)^{-1}e_m$.  We have
\begin{eqnarray}
f_{l,m}^* - f_m^* &=& ((L_l^*)^{-1}-(L^*)^{-1})e_m = (L_l^*)^{-1}(L^*-L_l^*)(L^*)^{-1}e_m.\nonumber
\end{eqnarray}
Now equation (\ref{asdfg}) implies $L-L_l = (L-I)(I-P_l)$, so that
\begin{eqnarray}
f_{l,m}^* - f_m^*  = (L_l^*)^{-1}(I-P_l)(L^*-I)(L^*)^{-1}e_m\nonumber.
\end{eqnarray}
Applying proposition \ref{asdf} yields
\begin{eqnarray}
\Arrowvert f_{l,m}^* - f_m^* \Arrowvert \leq \frac{1}{1-\delta}\Arrowvert (I-P_l)(L^*-I)(L^*)^{-1}e_m \Arrowvert,\nonumber
\end{eqnarray}
which for fixed $m$ goes to $0$ as $l \rightarrow \infty$.  We conclude $\lim_{l \rightarrow \infty} f_{l,m}^* = f_m^*$, which, upon passing to the Fourier transform, yields $\lim_{l \rightarrow \infty} G_{l,m} = G_m$. \end{proof}

\noindent Proof of Theorem \ref{levinsonlike}.\\

\noindent We see that $\delta_{nm} = \langle G_{l,m} , S_{l,n} \rangle$, where $S_{l,n} (t) =\mathrm{sinc}(\pi(t-t_n))$ when $|n| \leq l$ and $ S_{l,n} (t) =\mathrm{sinc}(\pi(t-n))$ when $|m| > l$.
Without loss of generality, let $|m|<l$. Equation (\ref{repker}) implies that $G_{l,m}(k)=0$ when $|k|>l$.  By the WKS theorem we have
\begin{eqnarray}
G_{l,m} (t) & = & \sum_{k=-l}^{k=l} G_{l,m} (k) \mathrm{sinc}(\pi(t-k)) = \Big( \sum_{k=-l}^{k=l} \frac{ t G_{l,m} (k)}{k-t} \Big) \mathrm{sinc}(\pi t)\nonumber\\ & = & \frac{w_l(t)}{\prod_{k=1}^{l} (k-t)(-k-t)} \mathrm{sinc}(\pi t),\nonumber
\end{eqnarray}
where $w_l$ is a polynomial of degree at most $2l$.  Noting that $$\mathrm{sinc}(\pi t) = \prod_{k=1}^{\infty} \big( 1-\frac{t^2}{k^2} \big) \quad \mathrm{and} \quad \prod_{k=1}^{l} (k-t)(-k-t) = (-1)^l (l!)^2 \prod_{k=1}^{l} \big( 1-\frac{t^2}{k^2} \big),$$ we have
\begin{eqnarray}
G_{l,m} (t) = \frac{(-1)^l w_l (t)}{(l!)^2} \prod_{k=l+1}^{\infty} \big( 1-\frac{t^2}{k^2} \big).\nonumber
\end{eqnarray}
Again by equation~(\ref{repker}), $\delta_{nm} = G_{l,m} (t_n)$ when $|n| \leq l$ so that $$\delta_{nm} = \frac{(-1)^l}{(l!)^2} w_l (t_n) \prod_{k=l+1}^{\infty} \big( 1-\frac{t_n^2}{k^2} \big).$$  This determines the zeroes of $w_l$.  We deduce that $$w_l(t) = \frac{c_l \prod_{k=1}^{k=l} (t-t_k)(t-t_{-k})}{t-t_m}$$ for some constant $c_l$.  Absorbing constants, we have
\begin{equation}
G_{l,m} (t)  =  \frac{c_l H_l (t)}{t-t_m},\nonumber
\end{equation}
where
\begin{equation}
H_l (t) := (t-t_0) \prod_{k=1}^{l} \big(1-\frac{t}{t_k} \big) \big( 1-\frac{t}{t_{-k}} \big)  \prod_{l+1}^{\infty} \big( 1-\frac{t^2}{k^2} \big).\nonumber
\end{equation}

\noindent Now $0 = H_l (t_m)$, so $G_{l,m} (t) = c_l \frac{H_l (t) -H_l (t_m)}{t-t_m}$.  Taking limits, $c_l = \frac{1}{(H_l)'(t_m)}$.  This yields 
\begin{equation} \\
G_{l,m} (t) = \frac{H_l (t)}{(t-t_m) H_l'(t_m) }.\nonumber
\end{equation}
Define $$H(t) = (t-t_0)\prod_{k=1}^{\infty}  \big(1 - \frac{t}{t_k} \big) \big( 1 - \frac{t}{t_{-k}} \big).$$  Basic complex analysis shows that $H$ is entire, and $H_l \rightarrow H$ and $H_l' \rightarrow H'$ uniformly on compact subsets of $\mathbb{C}$.
Furthermore, $H'(t_k) \neq 0$ for all $k$, since each $t_k$ is a zero of $H$ of multiplicity one. Together we have
\begin{equation}
\lim_{l \rightarrow \infty} G_{l,m}(t) = \frac{H(t)}{(t-t_m)H'(t_m)}, \quad t \in \mathbb{R}.\nonumber
\end{equation}
By the foregoing lemma, $G_{l,m} \rightarrow G_m$. Observing that convergence in $PW_{[-\pi,\pi]}$ implies pointwise convergence yields the desired result.\\

\noindent Levinson proved a version of Theorem \ref{levinsonlike} in the case where $(t_n)_{n \in \mathbb{Z}}$ is a Kadec sequence.  His original proof is found in \cite[pages 47-67]{L}).  We recall that if $(f_n)_n$ is a Riesz basis arising from a Kadec sequence, then the preframe operator $L$ satisfies $\Arrowvert I-L \Arrowvert <1$.  Levinson's theorem is then recovered from Theorem \ref{levinsonlike}.


\end{document}